\documentclass{amsart}
\usepackage[colorlinks]{hyperref}
\AtBeginDocument{
  \hypersetup{
    linkcolor=blue,
    citecolor=green,
  }
}
\usepackage{mathscinet}
\usepackage{mathtools}
\usepackage{bm}
\usepackage{cleveref}
\crefname{equation}{}{}

\makeatletter
\let\c@author\relax
\makeatother
\usepackage[backend=bibtex,style=numeric,sorting=nyt,
giveninits=true,isbn=false,url=false]{biblatex}
\renewbibmacro{in:}{\ifentrytype{article}{}{\printtext{\bibstring{in}\intitlepunct}}}
\DeclareFieldFormat*{title}{#1}
\usepackage[mathlines,pagewise]{lineno}
\newcommand*\patchAmsMathEnvironmentForLineno[1]{%
  \expandafter\let\csname old#1\expandafter\endcsname\csname #1\endcsname
  \expandafter\let\csname oldend#1\expandafter\endcsname\csname end#1\endcsname
  \renewenvironment{#1}%
     {\linenomath\csname old#1\endcsname}%
     {\csname oldend#1\endcsname\endlinenomath}}%
\newcommand*\patchBothAmsMathEnvironmentsForLineno[1]{%
  \patchAmsMathEnvironmentForLineno{#1}%
  \patchAmsMathEnvironmentForLineno{#1*}}%
\AtBeginDocument{%
\patchBothAmsMathEnvironmentsForLineno{equation}%
\patchBothAmsMathEnvironmentsForLineno{align}%
\patchBothAmsMathEnvironmentsForLineno{flalign}%
\patchBothAmsMathEnvironmentsForLineno{alignat}%
\patchBothAmsMathEnvironmentsForLineno{gather}%
\patchBothAmsMathEnvironmentsForLineno{multline}%
}

\usepackage{aliascnt}
\usepackage{cleveref}

\newtheorem{theorem}{Theorem}[section]
\Crefname{theorem}{Theorem}{Theorems}

\newaliascnt{lemma}{theorem}      
\newtheorem{lemma}[lemma]{Lemma}  
\aliascntresetthe{lemma}          
\Crefname{lemma}{Lemma}{Lemmas}

\newaliascnt{corollary}{theorem}      
\newtheorem{corollary}[corollary]{Corollary}  
\aliascntresetthe{corollary}          
\Crefname{corollary}{Corollary}{Corollaries}

\newaliascnt{proposition}{theorem}      
\aliascntresetthe{proposition}          
\Crefname{proposition}{Proposition}{Propositions}

\theoremstyle{definition}
\newaliascnt{definition}{theorem}
\newtheorem{definition}[definition]{Definition}
\aliascntresetthe{definition}
\Crefname{definition}{Definition}{Definitions}

\newaliascnt{notation}{theorem}

\aliascntresetthe{notation}
\Crefname{notation}{Notation}{Notations}

\theoremstyle{remark}
\newaliascnt{remark}{theorem}
\newtheorem{remark}[remark]{Remark}
\aliascntresetthe{remark}
\Crefname{remark}{Remark}{Remarks}

\crefname{equation}{}{}
\numberwithin{equation}{section}

\begin{document}

\title[$BMO$ regularity]{Interior $BMO$ regularity for elliptic equations in divergence form}

\author{Yuanyuan Lian}
\address{Departamento de An\'{a}lisis Matem\'{a}tico, Instituto de Matem\'{a}ticas IMAG, Universidad de Granada}
\email{lianyuanyuan.hthk@gmail.com; yuanyuanlian@correo.ugr.es}

\thanks{This research has been supported by the Grants PID2020-117868GB-I00 and PID2023-150727NB-I00 of the MICIN/AEI.}

\subjclass[2020]{Primary 35B65, 35D30, 35J15}

\date{}

\keywords{Interior regularity; $BMO$ regularity; Pointwise regularity; Elliptic equations; Divergence-form equations}

\begin{abstract}
In this note, we establish the interior $BMO$ regularity of weak solutions to uniformly elliptic equations in divergence form. Moreover, the assumptions on the coefficients are nearly optimal.
\end{abstract}

\maketitle

\section{Introduction}
Consider the following elliptic equation in divergence form:
\begin{equation}\label{e1.1}
(a^{ij} u_i)_j=f \quad\mbox{ in}~~B_1,
\end{equation}
where $B_1 \subset \mathbb{R}^{n}$ ($n\geq 3$) denotes the unit ball and $a^{ij}$ is uniformly elliptic. According to the classical De Giorgi-Nash-Moser theory, if $f\in L^p(B_1)$ for some $p>n/2$, then $u\in C^{\alpha}(\bar{B}_{1/2})$ for some $0<\alpha<1$ (see \cite[Theorem 8.24]{MR1814364}). It is therefore natural to ask what happens if we only assume $f\in L^{n/2}(B_1)$?

Clearly, one cannot expect $u\in C^{\alpha}(\bar{B}_{1/2})$ or even $u\in L^{\infty}(\bar{B}_{1/2})$, not even for the Poisson's equation. This can be seen from the following example (inspired by \cite[Example 1]{MR3048265}):
\begin{equation*}
u(x)=\ln |\ln |x||, \quad f(x)=\frac{n-2}{|x|^2\ln|x|}-\frac{1}{|x|^2\ln^2|x|}, \quad \Delta u=f\quad\mbox{in}~~ B_{e^{-1}}.
\end{equation*}
On the other hand, for the Poisson's equation, by the $W^{2,n/2}$ regularity result (see \cite[Corollary 9.10]{MR1814364}), we have $u\in W^{2,n/2}(B_{1/2})$. Consequently, by the Sobolev embedding, $u\in BMO(B_{1/2})$ (see \cite[(v) on P. 111]{MR3099262}).

From the above analysis, the best regularity we can expect for \cref{e1.1} with $f\in L^{n/2}(B_1)$ is $u\in BMO(B_{1/2})$. In this note, we show that this regularity indeed holds. When $n=2$, a weak solution $u\in W^{1,2}=W^{1,n}$ is automatically embedded into $BMO$ automatically. Hence, we only consider the case $n\geq 3$.

It seems impossible to obtain $BMO$ regularity by the classical De Giorgi-Nash-Moser technique. Instead, we employ a perturbation argument. It is well-known that if $u$ belongs to the De Giorgi class, then $u\in C^{\alpha}$ for some $0<\alpha<1$. We regard equation \cref{e1.1} as a perturbation of a De Giorgi class and then establish the $BMO$ regularity in a manner analogous to the Schauder estimate. The observation that the $BMO$ regularity is in fact a pointwise regularity similar to the Schauder estimate is inspired by \cite{MR4688261}.

In this note, we consider the elliptic equation in a general form:
\begin{equation}\label{e.div-0}
(a^{ij}u_i+d^ju)_j+b^iu_i+cu=f-f^i_i \quad\mbox{in}~~B_1,
\end{equation}
where $\bm{a}=(a^{ij})_{n\times n}$, $\bm{b}=(b^1,...,b^n)$, $\bm{d}=(d^1,...,d^n)$, $\bm{f}=(f^1,...,f^n)$, and the Einstein summation convention is adopted.

Throughout this note, we always assume that the matrix $\bm{a}$ is uniformly elliptic with ellipticity constants $0<\lambda\leq \Lambda$, i.e.,
\begin{equation}\label{e1.2}
\lambda |\xi|^2\leq a^{ij}\xi_i\xi_j\leq \Lambda |\xi|^2,~\forall ~\xi\in \mathbb{R}^n
\end{equation}
and
\begin{equation}\label{e1.3}
 \bm{b},\bm{d}\in L^{n}(B_1),\quad c\in L^{\frac{n}{2}}(B_1) ,\quad f\in L^{\frac{2n}{n+2}}(B_1),\quad \bm{f}\in L^2(B_1).
\end{equation}
The assumptions \cref{e1.2} and \cref{e1.3} are probable the minimal requirements for establishing the Caccioppoli inequality, which serves as the most fundamental estimate for elliptic equations in divergence form. Our proof of the $BMO$ regularity is just based on the Caccioppoli inequality. Strictly speaking, the Caccioppoli inequality cannot be derived directly under the assumptions \cref{e1.2} and \cref{e1.3}; we need to assume that the corresponding norms are sufficiently small. However, this smallness assumption poses no difficulty for our argument.

We adopt the following notations, as in \cite{MR4688261}:
\begin{equation*}
f_{\Omega}= \frac{1}{|\Omega|}\int_{\Omega} f, \quad \|f\|^*_{L^p(\Omega)}= \left(\frac{1}{|\Omega|}\int_{\Omega} |f|^p\right)^{1/p},  ~1\leq p<\infty,
\end{equation*}
where $\Omega\subset \mathbb{R}^n$ is a bounded domain and $|\Omega|$ denotes its Lebesgue measure. Note that the scaling of $\|f\|^*_{L^{p}(\Omega)}$ is the same as that of $f$. This normalization simplifies the arguments involving scaling (e.g. \cref{e2.1}) and the applications of H\"{o}lder inequality (e.g. \cref{e1.4}).

Recall the definition of the $BMO$ space (bounded mean oscillation, see \cite[Chapter 6.3]{MR3099262} for example).
\begin{definition}\label{de1.1}
For $x_0\in \Omega$ and $r_0>0$, if
\begin{equation*}
   |f|_{*,p,x_0}:=\sup_{0<r<r_0} \| f-f_{B_r(x_0)\cap \Omega}\|^*_{L^p(B_r(x_0)\cap \Omega)}<+\infty,
\end{equation*}
we say that $f$ is $BMO_p$ at $x_0$ or $f\in BMO_p(x_0)$ (with radius $r_0$).

Let $\Omega'\subset \Omega$. If $f\in BMO_p(x_0)$ for any $x_0\in \Omega'$ with the same radius $r_0$ and
\begin{equation*}
 |f|_{*,p,\Omega'}:=\sup_{x\in \Omega'}|f|_{*,p,x} <+\infty,
\end{equation*}
we say that $f\in BMO_p(\Omega')$ (with radius $r_0$). Moreover, we endow $BMO_p(\Omega')$ with the following norm:
\begin{equation*}
  \|f\|_{BMO_p(\Omega')}:=\|f\|_{L^p(\Omega')}+|f|_{*,p,\Omega'}.
\end{equation*}
\end{definition}

\medskip

We remark that the $BMO$ spaces defined using different $L^p$ norms are equivalent (see \cite[Corollary on P. 144 ]{MR1232192} or \cite[Corollary 6.22]{MR3099262}). Hence, we may write $\|f\|_{BMO(\Omega)}$ instead of $\|f\|_{BMO_p(\Omega)}$. Throughout this note, we mainly employ the $L^2$ norm which is more appropriate for weak solutions.

Next, we introduce a notion of pointwise smoothness for functions. Recall that a function $\omega:\mathbb{R}_+\rightarrow \mathbb{R}_+$ is called a modulus of continuity if $\omega$ is nondecreasing and $\omega(r)\to 0$ as $r\to 0$.

\begin{definition}\label{d-f2}
Let $\Omega\subset \mathbb{R}^n$ be a bounded domain, $f:\Omega\to \mathbb{R}$ and $\omega$ be a modulus of continuity. We say that $f$ is $C_{p}^{-k}$ ($k\in \mathbb{R}_+,1\leq p<+\infty$) at $x_0\in \Omega$ or $f\in C_{p}^{-k}(x_0)$ if $n/p\geq k$ and there exists $0<r_0\leq 1$ such that
\begin{equation}\label{e.c-1}
\|f\|^*_{L^{p}(B_r(x_0)\cap \Omega)}\leq r^{-k}\omega(r), ~\forall ~0<r<r_0.
\end{equation}
Then define
\begin{equation*}
\|f\|_{C_{p}^{-k}(x_0)}=\inf \left\{\omega(r_0) \big | \cref{e.c-1} ~\mbox{holds with}~\omega\right\}.
\end{equation*}
If $f\in C_{p}^{-k}(x)$ for any $x\in \Omega'\subset \Omega$ with the same $r_0$ and
\begin{equation*}
  \|f\|_{C_{p}^{-k}(\bar{\Omega}')}:= \sup_{x\in \Omega'} \|f\|_{C_{p}^{-k,\alpha}(x)}<+\infty,
\end{equation*}
we say that $f\in C_{p}^{-k}(\bar{\Omega}')$.
\end{definition}

\medskip

\begin{remark}\label{r-df1.5}
The requirement $n/p\geq k$ ensures that \cref{e.c-1} is well defined. If $n/p<k$, then \cref{e.c-1} yields
\begin{equation*}
\|f\|_{L^{p}(B_r(x_0))}\leq K r^{n/p-k}\omega(r), ~\forall ~0<r<r_0,
\end{equation*}
which may be trivial since the left-hand tends to zero while the right-hand may tend to infinity as $r\to 0$.

By noting \cref{e1.3}, the following notations are well defined:
\begin{equation*}
f\in C_{\frac{2n}{n+2}}^{-2}(0),\quad \bm{f}\in C_{2}^{-1}(0),
\end{equation*}
and will be used in our theorems. Since we always assume $f\in L^{\frac{2n}{n+2}}(B_1)$ and $\bm{f}\in L^2(B_1)$, the subscripts are omitted when we state the pointwise smoothness of $f$ and $\bm{f}$, i.e., we write $f\in C^{-2}(0)$ and $\bm{f}\in C^{-1}(0)$ in place of $f\in C_{\frac{2n}{n+2}}^{-2}(0)$ and $\bm{f}\in C_2^{-1}(0)$ for simplicity.
\end{remark}

\begin{remark}\label{r-df2.2}
In this note, we assume, without loss of generality, that $r_0=1$ when using the above definition.
\end{remark}

\medskip

Now, we state our main results.
\begin{theorem}[\textbf{Pointwise $BMO$ regularity}]\label{th1.1}
Let $u\in W^{1,2}(B_1)$ be a weak solution of \Cref{e.div-0}. Suppose that
\begin{equation}\label{e1.0}
  \begin{aligned}
&\bm{b}\in L^{n}(B_1),\\
&\|c\|_{L^{n/2}(B_r)},~~\|\bm{d}\|_{L^{n}(B_r)}\leq \frac{\delta_0}{|\ln r|+1},~\forall ~0<r\leq 1,\\
&f \in C^{-2}(0),~~\bm{f} \in C^{-1}(0),
  \end{aligned}
\end{equation}
where $\delta_0>0$ (small) is universal.

Then $u\in BMO_2(0)$ and
\begin{equation}\label{e.Ca.esti-ell}
  |u|_{*,2,0}\leq C\left(\|u\|_{L^{2}(B_1)}+\|f\|_{C^{-2}(0)}+\|\bm{f}\|_{C^{-1}(0)}\right),
\end{equation}
where $C$ depends only on $n,\lambda, \Lambda$ and $\bm{b}$.
\end{theorem}

\medskip

\begin{remark}\label{re1.1}
A constant is said to be universal if it depends only on $n,\lambda$ and $\Lambda$.
\end{remark}

\begin{remark}\label{re1.1-2}
If $f\in L^{n/2}(B_1),\bm{f}\in L^{n}(B_1)$, then the assumptions on $f$ and $\bm{f}$ in \cref{e1.0} can be satisfied. Indeed, by H\"{o}lder inequality,
\begin{equation}\label{e1.4}
\|f\|^*_{L^{\frac{2n}{n+2}}(B_r)}\leq \|f\|^*_{L^{\frac{n}{2}}(B_r)}= C r^{-2}\|f\|_{L^{\frac{n}{2}}(B_r)},
\end{equation}
where $C$ depends only on $n$. Hence, $f\in C^{-2}(0)$. Similarly, one obtains $\bm{f}\in C^{-1}(0)$.

If $c\in L^p,\bm{d}\in L^{2p}$ for some $p>n/2$, then the assumptions for $c$ and $\bm{d}$ in \cref{e1.0} can be satisfied by a scaling argument. Take $c$ as an example, and suppose that $u$ is a weak solution of
\begin{equation*}
\Delta u+cu=0\quad\mbox{in}~~B_1.
\end{equation*}
For $\rho>0$, introduce the following scaling:
\begin{equation*}
y=\frac{x}{\rho}, \quad v(y)=u(x).
\end{equation*}
Then $v$ satisfies
\begin{equation*}
\Delta v+\tilde cv=0\quad\mbox{in}~~B_1,
\end{equation*}
where $\tilde{c}(y)=\rho^2c(x)$. By choosing $\rho$ sufficiently small, we have
\begin{equation*}
\|\tilde c\|_{L^{n/2}(B_r)}=\|c\|_{L^{n/2}(B_{\rho r})}\leq C \|c\|_{L^{p}(B_{\rho r})} (\rho r)^{2-n/p}
\leq \frac{\delta_0}{|\ln r|+1},~\forall ~0<r\leq 1,
\end{equation*}
where $C$ depends only on $n$. Therefore, $v\in BMO_2(0)$ and hence $u\in BMO_2(0)$.
\end{remark}

\begin{remark}\label{re1.2}
Let us make a few remarks on the necessity of  the assumptions in \cref{e1.0}. Consider the following example:
\begin{equation*}
u(x)=|x|^{-\alpha}\notin BMO_2(0), \quad 0<\alpha<1.
\end{equation*}
By a direct calculation,
\begin{equation*}
u_i=-\alpha|x|^{-\alpha-2}x_i, \quad \Delta u=\alpha(\alpha+2-n)|x|^{-\alpha-2}.
\end{equation*}
Hence,
\begin{equation*}
  \left\{
  \begin{aligned}
    &\Delta u+b^iu_i=0, \quad && b^i(x)=(\alpha+2-n)|x|^{-2}x_i;\\
    &\Delta u=f, \quad && f(x)=\alpha(\alpha+2-n)|x|^{-\alpha-2};\\
    &\Delta u=-f^i_i, \quad && f^i(x)=\alpha|x|^{-\alpha-2}x_i.
  \end{aligned}
  \right.
\end{equation*}
Take $\alpha$ small enough such that $u\in W^{1,2}(B_1)$, $f\in L^{2n/(n+2)}(B_1)$ and $\bm{f}\in L^2(B_1)$. On the other hand,
\begin{equation*}
\bm{b}\in L^p(B_1),~\forall ~p<n, \quad f\in C^{-2-\alpha}(0), \quad \bm{f}\in C^{-1-\alpha}(0).
\end{equation*}
Therefore, the assumptions on $\bm{b}$, $f$ and $\bm{f}$ are optimal since we can choose $\alpha$ arbitrary small.

Next, we consider the following example:
\begin{equation*}
u(x)=\frac{1}{2}(\ln|x|)^2\notin BMO_2(0).
\end{equation*}
By a direct calculation,
\begin{equation*}
u_i=(\ln |x|) |x|^{-2}x_i, \quad \Delta u=(n-2)(\ln |x|)|x|^{-2}+|x|^{-2}.
\end{equation*}
Then
\begin{equation*}
\Delta u+(d^iu)_i+cu=0\quad\mbox{in}~~B_1,
\end{equation*}
where
\begin{equation*}
d^i(x)=(\ln|x|)^{-1}|x|^{-2}x_i, \quad c(x)=-3(n-2)(\ln|x|)^{-1}|x|^{-2}-3(\ln|x|)^{-2}|x|^{-2}.
\end{equation*}
Hence,
\begin{equation*}
  \begin{aligned}
&\|c\|_{L^{n/2}(B_r)}\leq \frac{C}{|\ln r|^{1-\frac{2}{n}}},~\forall ~0<r<\frac{1}{2},\\
&\|\bm{d}\|_{L^{n}(B_r)}\leq \frac{C}{|\ln r|^{1-\frac{1}{n}}},~\forall ~0<r<\frac{1}{2}.
  \end{aligned}
\end{equation*}

Above analysis implies that the assumptions
\begin{equation*}
\|c\|_{L^{n/2}(B_r)},~~\|\bm{d}\|_{L^{n}(B_r)}\leq \frac{\delta_0}{|\ln r|+1},~\forall ~0<r\leq 1
\end{equation*}
cannot be weakened to
\begin{equation*}
\|c\|_{L^{n/2}(B_r)},~~\|\bm{d}\|_{L^{n}(B_r)}\leq \frac{\delta_0}{|\ln r|^{\alpha}+1},~\forall ~0<r\leq 1
\end{equation*}
for a fixed $0<\alpha<1$ since we can choose $n$ large enough such that $1-2/n>\alpha$. However, we do not know whether the small constant $\delta_0$ can be replaced by a general constant $C$.
\end{remark}

\medskip

From the pointwise $BMO$ regularity \Cref{th1.1}, we have the following local $BMO$ regularity:
\begin{corollary}[\textbf{Local $BMO$ regularity}]\label{co1.2}
Let $u\in W^{1,2}(B_1)$ be a weak solution of \Cref{e.div-0}. Suppose that
\begin{equation*}
  \begin{aligned}
&\bm{b}\in L^{n}(B_1),\\
&\|c\|_{L^{n/2}(B_r(x_0))},~~\|\bm{d}\|_{L^{n}(B_r(x_0))}\leq \frac{\delta_0}{|\ln r|+1},~\forall ~0<r\leq 1/2,
~\forall ~x_0\in B_{1/2},\\
&f \in C^{-2}(\bar B_{1/2}),~~\bm{f} \in C^{-1}(\bar B_{1/2}),
  \end{aligned}
\end{equation*}
where $\delta_0>0$ (small) is universal.

Then $u\in BMO(B_{1/2})$ and
\begin{equation*}
 \|u\|_{BMO(B_{1/2})}\leq C\left(\|u\|_{L^{2}(B_1)}+\|f\|_{C^{-2}(\bar B_{1/2})}+\|\bm{f}\|_{C^{-1}(\bar B_{1/2})}\right),
\end{equation*}
where $C$ depends only on $n,\lambda, \Lambda$ and $\bm{b}$.
\end{corollary}

As a special case of \Cref{co1.2}, we have
\begin{corollary}[\textbf{Local $BMO$ regularity}]\label{co1.1}
Let $u\in W^{1,2}(B_1)$ be a weak solution of \Cref{e.div-0}. Suppose that for some $p>n/2$,
\begin{equation*}
  \begin{aligned}
\bm{b}\in L^{n}(B_1), \quad c \in L^{p}(B_1), \quad \bm{d}\in L^{2p}(B_1), \quad
f \in L^{n/2}(B_1), \quad \bm{f} \in L^{n}(B_1).
  \end{aligned}
\end{equation*}

Then $u\in BMO(B_{1/2})$ and
\begin{equation*}
 \|u\|_{BMO(B_{1/2})}\leq C\left(\|u\|_{L^{2}(B_1)}+\|f\|_{L^{n/2}(B_1)}+\|\bm{f}\|_{L^{n}(B_1)}\right),
\end{equation*}
where $C$ depends only on $n,\lambda, \Lambda,p$ and $\bm{b}$.
\end{corollary}

\medskip

\begin{remark}\label{re1.3}
We thank the referee for \emph{Proceedings of the AMS} for pointing out that \Cref{co1.1} (without lower-order terms) is included in \cite[Corollary 8.1]{MR2900466}, where the result was proved by a different method.
\end{remark}
\medskip

\section{\texorpdfstring{$BMO$}{BMO} regularity}
We first describe the main idea briefly. We adopt a perturbation argument, similar to the approach used in proving Schauder regularity (i.e., the $C^{k,\alpha}$ regularity). The main obstacle is that the set of weak solutions is not closed under weak limits, since the matrix $\bm{a}$ is only uniformly elliptic. This issue is overcome by noting that the De Giorgi class is closed under weak limits (see \Cref{le2.0-2}). It is well-known that functions in the De Giorgi class enjoy $C^{\alpha}$ regularity, which is stronger than $BMO$ regularity. Therefore, the perturbation argument can be applied in this context.

We next recall the definition of the De Giorgi class and its H\"{o}lder regularity, which is its most fundamental property (see \cite{MR0093649, MR3592445,MR3839844} and \cite[Chapter 10]{MR4689649}).
\begin{definition}\label{de2.1}
Let $1< p<+\infty, \gamma>0$ and $\Omega\subset \mathbb{R}^n$ be a bounded domain. We say $u\in [DG]_p^+(\Omega;\gamma)$ if $u\in W^{1,p}(\Omega)$ and satisfies
\begin{equation}\label{e2.2}
\int_{B_r(x)} |D(u-k)^+|^p\leq \frac{\gamma}{(R-r)^p}\int_{B_R(x)} |(u-k)^+|^p,~\forall ~k\in \mathbb{R},~~B_r(x)\subset B_R(x)\subset \Omega.
\end{equation}
Here, we use the notation $a^+:=\max(0,a)$ and $a^-:=-\min(0,a)$. Similarly, we can define $[DG]_p^-(\Omega;\gamma)$. Finally, let
\begin{equation*}
[DG]_p(\Omega;\gamma)=[DG]_p^+(\Omega;\gamma)\cap [DG]_p^-(\Omega;\gamma).
\end{equation*}
\end{definition}
\medskip

\begin{lemma}\label{le2.0}
Let $u\in [DG]_p(B_1;\gamma)$. Then $u\in C^{\bar \alpha}(\bar{B}_{1/2})$ and
\begin{equation}\label{e2.3}
\|u\|_{C^{\bar \alpha}(\bar{B}_{1/2})}\leq \bar C\|u\|_{L^p(B_1)},
\end{equation}
where $0<\bar \alpha<1$ and $\bar C\geq 1$ depend only on $n,p$ and $\gamma$.
\end{lemma}
\medskip

Now, we show that the De Giorgi class is closed for weak solutions.
\begin{lemma}\label{le2.0-2}
Let $u_m\in W^{1,2}(B_1)$ be a sequence of weak solutions of
\begin{equation*}
(a_m^{ij}u_{m,i}+d_m^ju_m)_j+b_m^iu_{m,i}+c_mu_m=f_m-f^i_{m,i} \quad\mbox{in}~~B_1.
\end{equation*}
Suppose that $\|u_m\|_{L^{2}(B_1)}\leq 1$ and as $m\to \infty$,
\begin{equation}\label{e2.5}
\|\bm{b}_m\|_{L^{n}(B_1)},\|c_m\|_{L^{\frac{n}{2}}(B_1)},\|\bm{d}_m\|_{L^{n}(B_1)},\|f_m\|_{L^{\frac{2n}{n+2}}(B_1)},
\|\bm{f}_m\|_{L^{2}(B_1)}\rightarrow 0.
\end{equation}
Then there exists $\bar{u}\in [DG]_2(B_{3/4};\gamma)$ such that up to a subsequence,
\begin{equation}\label{e2.4}
u_m\to \bar{u}\quad\mbox{in}~~L^2(B_{3/4}), \quad Du_m\to D\bar{u}\quad\mbox{weakly in}~~L^2(B_{3/4}),
\end{equation}
where $\gamma$ is universal.
\end{lemma}
\proof By the assumption \cref{e2.5} and the Caccioppoli inequality (see \cite[Chapter 3, (4.39) and (4.41)]{MR0244627}), for any $\Omega'\subset\subset B_1$,
\begin{equation}\label{e6}
\|u_m\|_{W^{1,2}(\Omega')}\leq C(\|u_m\|_{L^{2}(B_{1})}+\|f_m\|_{L^{\frac{2n}{n+2}}(B_{1})}+\|\bm{f}_m\|_{L^{2}(B_{1})})\leq C,
\end{equation}
where $C$ depends only on $n,\lambda,\Lambda$ and $\Omega'$. By the Sobolev imbedding (see \cite[Theorem 3.26]{MR3099262}), there exist $\bar u\in W_{loc}^{1,2}(B_1)$ and a subsequence of $u_m$ (denoted by $u_m$ again) such that $u_m\rightarrow \bar u$ in $W_{loc}^{1,2}(B_1)$ weakly and in $L^2_{loc} (B_1)$ strongly.

Next, we show that $\bar{u}\in [DG]_2(B_{3/4};\gamma)$ for some universal constant $\gamma$. We only prove $\bar{u}\in [DG]_2^+(B_{3/4};\gamma)$ since the proof of $\bar{u}\in [DG]_2^-(B_{3/4};\gamma)$ is similar. For any $k\in \mathbb{R}$ and any $B_r(x_0)\subset B_R(x_0)\subset B_{3/4}$, we choose a nonnegative smooth cut-off function $\eta\in C^{\infty}_c(B_R(x_0))$ with $\eta\equiv 1$ in $B_r(x_0)$ and $\|D\eta\|_{L^{\infty}(B_1)}\leq C|R-r|^{-1}$, where $C$ depends only on $n$. Take $\eta^2(u_m-k)^+$ as the test function for $u_m$ and we have
\begin{equation}\label{e2.6}
  \begin{aligned}
\int_{B_1} &a^{ij}_mu_{m,i}(\eta^2(u_m-k)^+_j+2\eta\eta_j(u_m-k)^+)\\
=&-\int_{B_1} d^j_mu_m(\eta^2(u_m-k)^+_j+2\eta\eta_j(u_m-k)^+)\\
&+\int_{B_1} b^i_mu_{m,i}\eta^2(u_m-k)^++c_mu_m\eta^2(u_m-k)^+\\
&-\int_{B_1}f_m\eta^2(u_m-k)^++f^j_m(\eta^2(u_m-k)^+_j+2\eta\eta_j(u_m-k)^+)\\
:=&I_m.
  \end{aligned}
\end{equation}

Note that
\begin{equation*}
Du_m=D(u_m-k)^+\quad\mbox{if}~~u_m>k.
\end{equation*}
Hence,
\begin{equation*}
\int_{B_1}a^{ij}_mu_{m,i}\eta^2(u_m-k)^+_j=\int_{B_1}\eta^2a^{ij}_m(u_m-k)^+_i(u_m-k)^+_j
\geq \lambda\int_{B_r(x_0)} |D(u_m-k)^+|^2.
\end{equation*}
In addition, by the Young inequality,
\begin{equation*}
  \begin{aligned}
&2\left|\int_{B_1} a^{ij}_mu_{m,i}\eta\eta_j(u_m-k)^+\right|\\
&\leq \frac{\lambda}{2}\int_{B_1} \eta^2a^{ij}_m(u_m-k)^+_i(u_m-k)^+_j+C\int_{B_1} |(u_m-k)^+|^2a^{ij}_m\eta_i\eta_j\\
&\leq \frac{\lambda}{2}\int_{B_1} \eta^2a^{ij}_m(u_m-k)^+_i(u_m-k)^+_j+\frac{C}{(R-r)^2}\int_{B_R(x_0)}|(u_m-k)^+|^2.\\
  \end{aligned}
\end{equation*}
Thus, \cref{e2.6} reduces to
\begin{equation}\label{e2.8}
\frac{\lambda}{2}\int_{B_r(x_0)} |D(u_m-k)^+|^2\leq \frac{C}{(R-r)^2}\int_{B_R(x_0)}|(u_m-k)^+|^2+I_m.
\end{equation}

By the assumption \cref{e2.5}, $I_m\to 0$ as $m\to \infty$.
With the aid of the weak lower semi-continuity of the $L^2$ norm, by letting $m\to \infty$ in \cref{e2.8}, we have
\begin{equation*}
  \begin{aligned}
\int_{B_r(x_0)} |D(\bar u-k)^+|^2\leq& \liminf_{m\to \infty} \int_{B_r(x_0)} |D(u_m-k)^+|^2\\
\leq&\lim_{m\to \infty}\frac{\gamma}{(R-r)^2}\int_{B_R(x_0)}|(u_m-k)^+|^2+\lim_{m\to \infty}I_m\\
=& \frac{\gamma}{(R-r)^2}\int_{B_R(x_0)}|(\bar{u}-k)^+|^2,
  \end{aligned}
\end{equation*}
where $\gamma$ is universal. That is, $\bar{u}\in [DG]_2^+(B_{3/4};\gamma)$.~\qed~\\

Now, we can prove the key step towards the $BMO$ regularity.
\begin{lemma}\label{l-1}
Let $u\in W^{1,2}(B_1)$ be a weak solution of \cref{e.div-0}. Suppose that
\begin{equation*}
  \begin{aligned}
&\|u\|^*_{L^2(B_1)}\leq 1, \quad \|\bm{b}\|^*_{L^n(B_1)}\leq \delta, \quad \|c\|^*_{L^{n/2}(B_1)}\leq \delta,
, \quad \|\bm{d}\|^*_{L^n(B_1)}\leq \delta\\
    &\|f\|^*_{L^{\frac{2n}{n+2}}(B_1)}\leq \delta, \quad \|\bm{f}\|^*_{L^2(B_1)}\leq \delta,
  \end{aligned}
\end{equation*}
where $0<\delta<1$ is a small universal constant.

Then there exists a constant $P$ such that
\begin{equation}
  \begin{aligned}
    &\|u-P\|^*_{L^2(B_{\eta})}\leq 1,\\
    &|P|\leq \bar{C},
  \end{aligned}
\end{equation}
where $0<\eta<1/4$ and $\bar{C}$ (is from \Cref{le2.0}) are universal.
\end{lemma}
\proof We prove the lemma by contradiction. Suppose that the lemma is false. Then there exist sequences of $u_m,\bm{a}_m,\bm{b}_m,c_m,\bm{d}_m, f_m$ and $\bm{f}_m$ such that for any $m$, $u_m$ is the weak solution of
\begin{equation*}
(a_m^{ij}u_{m,i}+d_m^ju_m)_j+b_m^iu_{m,i}+c_mu_m=f_m-f^i_{m,i} \quad\mbox{in}~~B_1.
\end{equation*}
Moreover,
\begin{equation*}
  \begin{aligned}
&\|u_m\|^*_{L^2(B_1)}\leq 1, \quad \|\bm{b}_m\|^*_{L^n(B_1)},~~\|c_m\|^*_{L^{n/2}(B_1)},~~ \|\bm{d}_m\|^*_{L^n(B_1)}\leq 1/m,\\
    &\|f_m\|^*_{L^{\frac{2n}{n+2}}(B_1)},~~ \|\bm{f}_m\|^*_{L^2(B_1)}\leq 1/m.
  \end{aligned}
\end{equation*}
Furthermore, for any constant $P$ with $|P|\leq \bar{C}$, we have
\begin{equation}\label{e2}
\|u_m-P\|^*_{L^2(B_{\eta})}> 1,
\end{equation}
where $0<\eta<1/4$ is to be specified later.

From \Cref{le2.0-2}, there exist $\bar u\in [DG]_2(B_{3/4};\gamma)$ for some universal constant $\gamma$ and a subsequence of $u_m$ (denoted by $u_m$ again) such that $u_m\rightarrow \bar u$ in $W_{loc}^{1,2}(B_1)$ weakly and in $L^2_{loc} (B_1)$ strongly. By the H\"{o}lder regularity for the De Giorgi class (see \Cref{le2.0}), there exists a constant $\bar{P}$such that
\begin{equation*}
  \begin{aligned}
    &\|\bar u-\bar{P}\|^*_{L^2(B_{r})}\leq \bar C\|\bar u\|_{L^{2}(B_{3/4})}r^{\bar \alpha}
    \leq \bar Cr^{\bar \alpha}, ~\forall ~0<r<1/2,\\
    &|\bar{P}|\leq \bar C.
  \end{aligned}
\end{equation*}
Take $0<\eta<1/4$ small enough such that
\begin{equation*}
  \eta^{\bar \alpha} \bar{C}\leq 1/2.
\end{equation*}
Thus,
\begin{equation}\label{e3}
 \|\bar u-\bar{P}\|^*_{L^2(B_{\eta})}\leq 1/2.
\end{equation}

By \cref{e2},
\begin{equation*}
\|u_m-\bar P\|^*_{L^2(B_{\eta})}> 1.
\end{equation*}
Let $m\rightarrow \infty$ and we have
\begin{equation*}
\|\bar u-\bar{P}\|^*_{L^2(B_{\eta})}\geq 1,
\end{equation*}
which contradicts \cref{e3}. \qed~\\

Now, we give the scaling argument.
\begin{lemma}\label{t-1}
Let $0<\delta<1$ be as in \Cref{l-1} and $u\in W^{1,2}(B_1)$ be a weak solution of \cref{e.div-0}. Suppose that
\begin{equation}\label{e2.10}
  \begin{aligned}
&\|u\|^*_{L^2(B_1)}\leq 1, \quad \|\bm{b}\|^*_{L^n(B_1)}\leq \delta, \quad \\
&\|c\|^*_{L^{n/2}(B_r)}\leq \frac{\delta}{2\bar{C}r^2(|\ln r|+1)},~\forall ~0<r\leq 1,\\
&\|\bm{d}\|^*_{L^{n}(B_r)}\leq \frac{\delta}{2\bar{C}r(|\ln r|+1)},~\forall ~0<r\leq 1,\\
&\|f\|_{C^{-2}(0)}\leq \frac{\delta}{2}, \quad \|\bm{f}\|_{C^{-1}(0)}\leq \frac{\delta}{2}.
  \end{aligned}
\end{equation}

Then $u\in BMO_2(0)$ and
\begin{equation*}
  |u|_{*,2,0}\leq C,
\end{equation*}
where $C$ is universal.
\end{lemma}
\proof To prove $u\in BMO_2(0)$, we only need to prove the following: there exists a sequence of constants $P_m$ such that for any $m\geq 0$,
\begin{equation}\label{e4}
  \begin{aligned}
    &\|u-P_m\|^*_{L^2(B_{\eta^m})}\leq 1,\\
    &|P_m-P_{m-1}|\leq \bar{C},
  \end{aligned}
\end{equation}
where $0<\eta<1/4$ and $\bar{C}$ are as in \Cref{l-1}.

The first inequality in \cref{e4} implies $u\in BMO_2(0)$. Indeed, for any $0<r<1$, there exists $m\geq 0$ such that $\eta^{m+1}\leq r<\eta^m$. Recall
\begin{equation}\label{e7}
 \|u-u_{B_r}\|^*_{L^2(B_{r})}=\inf_{A\in \mathbb{R}}\|u-A\|^*_{L^2(B_{r})}.
\end{equation}
Then with the aid of \Cref{e4}, we have
\begin{equation*}
\begin{aligned}
 \|u-u_{B_r}\|^*_{L^2(B_{r})} \leq&\|u-P_{m}\|^*_{L^2(B_{r})}
 \leq\frac{1}{\eta}\|u-P_{m}\|^*_{L^2(B_{\eta^{m}})} \leq \frac{1}{\eta}.
\end{aligned}
\end{equation*}
That is, $u\in BMO_2(0)$.

In the following, we prove \cref{e4} by induction. For $m=0$, by setting $P_0\equiv P_{-1}\equiv 0$, the conclusion holds clearly. Suppose that the conclusion holds for $m$ and we need to prove that it holds for $m+1$.

Let $r=\eta^m$, $y=x/r$ and
\begin{equation*}
  \tilde{u}(y)=u(x)-P_m.
\end{equation*}
Then $\tilde{u}$ is a weak solution of
\begin{equation*}
 (\tilde a^{ij}\tilde u_i+\tilde d^j\tilde u)_j+\tilde b^i\tilde u_i+\tilde c \tilde u=\tilde f-\tilde f^i_i \quad \mbox{in}~~B_1,
\end{equation*}
where
\begin{equation*}
  \begin{aligned}
&\tilde a^{ij}(y)=a^{ij}(x), \quad \tilde b^i(y)=rb^i(x), \quad \tilde{c}(y)=r^2c(x), \quad \tilde{d}^i(y)=rd^i(x),\\
&\tilde{f}(y)=r^2f(x)-r^2c(x)P_m, \quad \tilde{f}^i(y)=rf^i(x)+rd^i(x)P_m.
  \end{aligned}
\end{equation*}

By induction,
\begin{equation}\label{e2.9}
|P_m|\leq \sum_{i=1}^{m}|P_i-P_{i-1}|\leq m\bar{C}=\frac{\bar{C}}{|\ln \eta|}|\ln r|\leq \bar{C}|\ln r|
\end{equation}
and
\begin{equation}\label{e2.1}
\|\tilde{u}\|^*_{L^2(B_1)}=\|u-P_m\|^*_{L^2(B_{\eta^m})}\leq 1.
\end{equation}
In addition, by the assumptions \cref{e2.10},
\begin{equation*}
  \begin{aligned}
&\|\bm{\tilde b}\|^*_{L^n(B_1)}=r\|\bm{b}\|^*_{L^n(B_r)}=\|\bm{b}\|_{L^n(B_r)}\leq \|\bm{b}\|_{L^n(B_1)}\leq \delta,\\
&\|\tilde c\|^*_{L^{n/2}(B_1)}=r^2\|c\|^*_{L^{n/2}(B_r)}\leq \delta,\\
&\|\bm{\tilde d}\|^*_{L^n(B_1)}=r\|\bm{d}\|^*_{L^n(B_r)}\leq \delta.
  \end{aligned}
\end{equation*}
Similarly, by \cref{e2.10} and \cref{e2.9},
\begin{equation*}
  \begin{aligned}
\|\tilde f\|^*_{L^{\frac{2n}{n+2}}(B_1)}\leq & r^2\|f\|^*_{L^{\frac{2n}{n+2}}(B_r)}+r^2|P_m|\|c\|^*_{L^{\frac{2n}{n+2}}(B_r)}\leq \frac{\delta}{2}
+r^2|P_m|\|c\|^*_{L^{n/2}(B_r)} \\
\leq& \frac{\delta}{2}+\frac{\delta}{2}= \delta,\\
\|\bm{\tilde f}\|^*_{L^2(B_1)}\leq& r\|\bm{f}\|^*_{L^2(B_r)}+r|P_m|\|\bm{d}\|^*_{L^2(B_r)}\leq\frac{\delta}{2}+ r|P_m|\|\bm{d}\|^*_{L^n(B_r)}\\
\leq& \frac{\delta}{2}+\frac{\delta}{2}= \delta.
  \end{aligned}
\end{equation*}
Therefore, the assumptions of \Cref{l-1} are satisfied. By \Cref{l-1}, there exists a constant $\tilde P$ such that
\begin{equation*}
  \begin{aligned}
    &\|\tilde{u}-\tilde P\|^*_{L^2(B_{\eta})}\leq 1,\\
    &|\tilde P|\leq \bar{C}.
  \end{aligned}
\end{equation*}
By rescaling back to $u$ with $P_{m+1}=P_m+\tilde P$, \cref{e4} holds for $m+1$. By induction, the proof is complete.\qed~\\

Now, we show that \Cref{t-1} implies the $BMO$ regularity.~\\
\noindent\textbf{Proof of \Cref{th1.1}.} We just need to make some normalization such that the assumptions of \Cref{t-1} are satisfied. For $0<\rho <1/4$, consider
\begin{equation*}
y=\frac{x}{\rho}, \quad \tilde{u}(y)=\frac{u(x)}{K}, \quad K=\|u\|^*_{L^2(B_{\rho})}+\frac{2}{\delta}\|f\|^*_{C^{-2}(0)}+\frac{2}{\delta}\|\bm{f}\|_{C^{-1}(0)},
\end{equation*}
where $\delta$ is as in \Cref{l-1}. Then $\tilde{u}$ is a weak solution of
\begin{equation*}
 (\tilde a^{ij}\tilde u_i+\tilde d^j\tilde u)_j+\tilde b^i\tilde u_i+\tilde c \tilde u=\tilde f-\tilde f^i_i \quad\mbox{in}~~B_1,
\end{equation*}
where
\begin{equation*}
  \begin{aligned}
&\tilde a^{ij}(y)=a^{ij}(x), \quad \tilde b^i(y)=\rho b^i(x), \quad \tilde{c}(y)=\rho^2c(x), \quad \tilde{d}^i(y)=\rho d^i(x),\\
&\tilde{f}(y)=\frac{1}{K}\rho^2f(x) , \quad \tilde{f}^i(y)=\frac{1}{K}\rho f^i(x).
  \end{aligned}
\end{equation*}

First,
\begin{equation*}
\|\tilde u\|^*_{L^2(B_1)}=\frac{1}{K}\|u\|^*_{L^2(B_{\rho})}\leq 1.
\end{equation*}
Next, by the assumptions of \Cref{th1.1},
\begin{equation*}
  \begin{aligned}
&\|\bm{\tilde b}\|^*_{L^n(B_1)}=\rho\|\bm{b}\|^*_{L^n(B_\rho)}=\|\bm{b}\|_{L^n(B_\rho)},\\
&\|\tilde c\|^*_{L^{n/2}(B_r)}=\rho^2\|c\|^*_{L^{n/2}(B_{\rho r})}\leq \frac{\delta_0}{r^2\left(|\ln \rho r|+1\right)}\leq \frac{\delta_0}{r^2\left(|\ln r|+1\right)}, ~\forall ~ 0<r<1,\\
&\|\bm{\tilde d}\|^*_{L^n(B_r)}=\rho\|\bm{d}\|^*_{L^n(B_{\rho r})}\leq \frac{\delta_0}{r\left(|\ln \rho r|+1\right)}\leq \frac{\delta_0}{r\left(|\ln r|+1\right)}, ~\forall ~ 0<r<1,\\
&\|\tilde f\|_{C^{-2}(0)}\leq  \frac{\delta}{2}, \quad \|\bm{\tilde f}\|_{C^{-2}(0)}\leq  \frac{\delta}{2}.
  \end{aligned}
\end{equation*}
Hence, by taking $\rho$ small enough (depending on $\bm{b}$) and $\delta_0$ small enough ($\delta_0\leq \delta/(2\bar{C})$), the assumptions of \Cref{t-1} are satisfied. From \Cref{t-1}, $\tilde{u}\in BMO_2(0)$. By rescaling back to $u$, we have $u\in BMO_2(0)$ and \cref{e.Ca.esti-ell} holds.~\qed

\section*{Declarations}
\noindent \textbf{Ethical Approval } not applicable.

\medskip

\noindent \textbf{Competing interests} The author declares that the author has no competing interests as defined by Springer, or other interests that might be perceived to influence the results and/or discussion reported in this paper.

\medskip

\noindent \textbf{Availability of data and materials} not applicable.

\printbibliography
\end{document}